\documentclass[draft,11pt,english]{article}
\usepackage{amsfonts,amsmath,amsthm,amscd,amssymb,latexsym}

    \usepackage[T2A]{fontenc}
    \usepackage[cp1251]{inputenc}
    \usepackage[russian]{babel}

\def\Xint#1{\mathchoice
   {\XXint\displaystyle\textstyle{#1}}%
   {\XXint\textstyle\scriptstyle{#1}}%
   {\XXint\scriptstyle\scriptscriptstyle{#1}}%
   {\XXint\scriptscriptstyle\scriptscriptstyle{#1}}%
   \!\int}
\def\XXint#1#2#3{{\setbox0=\hbox{$#1{#2#3}{\int}$}
     \vcenter{\hbox{$#2#3$}}\kern-.5\wd0}}
\def\dashint{\Xint-}

\begin{document}

\title{К теории соболевких отображений на комплексной плоскости}

\title{Foward the theory of Sobolev's mappings on the plane }

\author{ Руслан Р. Салимов}

\date{11.06.2010}

\theoremstyle{plain}
\newtheorem{theorem}{Теорема}[section]
\newtheorem{lemma}{Лемма}[section]
\newtheorem{proposition}{Предложение}[section]
\newtheorem{corollary}{Следствие}[section]
\newtheorem{definition}{Определение}[section]
\theoremstyle{definition}
\newtheorem{example}{Пример}[section]
\newtheorem{remark}{Замечание}[section]
\newcommand{\keywords}{\textbf{Key words.  }\medskip}
\newcommand{\subjclass}{\textbf{MSC 2000. }\medskip}
\renewcommand{\abstract}{\textbf{Аннотация.  }\medskip}
\numberwithin{equation}{section}

\maketitle

\begin{abstract}

The  paper is devoted to the study of homeomoephisms with finite distortion on the plane with use of the modulus techniques.  
\end{abstract}

В статье изучаются гомеоморфизмы с конечным искажением на плоскости с использованием модульной техники.

\section{Введение}
Непрерывное отображение $\gamma$ открытого подмножества $\Delta$
действительной оси ${\Bbb R}$ или окружности в $D$ называется {\bf
штриховой линией}, см., например, раздел 6.3 в \cite{MRSY}.
Напомним, что любое открытое множество $\Delta$ в ${\Bbb R}$ состоит
из счетного набора попарно непересекающихся интервалов. Это дает
мотивировку для термина "штриховая линия".

Пусть задано семейство $\Gamma$ штриховых линий $\gamma$ в
комплексной плоскости ${\Bbb C}$. Борелевскую функцию $\varrho:{\Bbb
C}\to[0,\infty]$ называют {\bf допустимой} для $\Gamma$, пишут
$\varrho\in{\rm adm}\,\Gamma$, если
\begin{equation}\label{eq1.2KR} \int\limits_{\gamma}\varrho\,ds\
\geqslant\ 1\qquad\forall\ \gamma\in\Gamma.\end{equation}

Пусть $p\geq 1$. Тогда   {\bf
$p$--модулем} семейства $\Gamma$ называется величина
\begin{equation}\label{eq1.3KR}M_p(\Gamma)\ =\
\inf_{\varrho\in\mathrm{adm}\,\Gamma}\int\limits_{{\Bbb
C}}\varrho^p(z)\,dm(z)\end{equation} где $dm(z)$ соответствует мере
Лебега в ${\Bbb C}$. Говорят, что свойство $P$ имеет место для {\bf $p$--
п.в.} (почти всех) $\gamma\in\Gamma$, если подсемейство всех линий в
$\Gamma$, для которых $P$ не верно имеет нулевой $p$--модуль, ср.
\cite{Fu}. Также говорят, что измеримая по Лебегу функция
$\varrho:{\Bbb C}\to[0,\infty]$ является {\bf обобщенно допустимой}
для $\Gamma$, пишут $\varrho\in{\rm ext_p\,adm}\,\Gamma$, если
(\ref{eq1.2KR}) имеет место для $p$--п.в. $\gamma\in\Gamma$, см.,
например, раздел 9.2 в \cite{MRSY}.

\section{О емкостях }

Следуя работе \cite{MRV}, пару $\mathcal{E}=(A,C)$, где $A\subset\mathbb{C}$
-- открытое множество и $C$ -- непустое компактное множество,
содержащееся в $A$, называем {\it конденсатором}. Конденсатор $\mathcal{E}$
называется  {\it кольцевым конденсатором}, если $B=A\setminus C$ --
кольцо, т.е., если $B$ -- область, дополнение которой
$\overline{\mathbb{C}}\setminus B$ состоит в точности из двух
компонент. Конденсатор $\mathcal{E}$ называется {\it ограниченным
конденсатором}, если множество $A$ является ограниченным. Говорят
также, что конденсатор $\mathcal{E}=(A,C)$ лежит в области $D$, если $A\subset
D$. Очевидно, что если $f:D\to\mathbb{C}$ -- непрерывное, открытое
отображение и $\mathcal{E}=(A,C)$ -- конденсатор в $D$, то $(fA,fC)$ также
конденсатор в $fD$. Далее $f\mathcal{E}=(fA,fC)$.

\medskip

Пусть $\mathcal{E}=(A,C)$ -- конденсатор. Обозначим $C_0(A)$ через множество
непрерывных функций $u:A\to\mathbb{R}^1$ с компактным носителем.
$W_0(\mathcal{E})=W_0(A,C)$ -- семейство неотрицательных функций
$u:A\to\mathbb{R}^1$ таких, что 1) $u\in C_0(A)$, 2)
$u(x)\geqslant1$ для $x\in C$ и 3) $u$ принадлежит классу ${\rm
ACL}$ и пусть
\begin{equation}\label{eqks2.5}\vert\nabla
u\vert={\left(\sum\limits_{i=1}^2\,{\left(\partial_i u\right)}^2
\right)}^{1/2}.\end{equation}

При $p\geqslant1$ величину
\begin{equation}\label{eqks2.6}{\rm cap}_p\,\mathcal{E}={\rm cap}_p\,(A,C)=\inf\limits_{u\in W_0(\mathcal{E})}\,
\int\limits_{A}\,\vert\nabla u\vert^p\,dm(z)\end{equation} называют
{\it $p$-ёмкостью} конденсатора $\mathcal{E}$. В дальнейшем мы будем
использовать равенство
\begin{equation}\label{eqks2.7}
{\rm cap}_p\,\mathcal{E}=M_p(\Delta(\partial A,\partial C; A\setminus
C)),\end{equation} где для множеств $\mathcal{S}_1$, $\mathcal{S}_2$
и $\mathcal{S}_3$ в $\mathbb{C}$,
$\Delta(\mathcal{S}_1,\mathcal{S}_2;\mathcal{S}_3)$ обозначает
семейство всех непрерывных кривых, соединяющих $\mathcal{S}_1$ и
$\mathcal{S}_2$ в $\mathcal{S}_3$, см. \cite{G}, \cite{H} и
\cite{Sh}. Емкости в контексте теории отображений хорошо отражены в
монографии \cite{GR}.

Известно, что при $p\geqslant1$ \begin{equation}\label{eqks2.8} {\rm
cap}_p\,\mathcal{E}\geqslant\frac{\left(\inf
m_{n-1}\,\sigma\right)^{p}}{\left[m(A\setminus
C)\right]^{p-1}},\end{equation} где $l(\sigma)$ -- длина кривой, где $\sigma$-гладкая (бесконечно дифференцируемая) кривая, которая является  границей $\sigma=\partial U$ ограниченного открытого
множества $U$, содержащего $C$ и содержащегося вместе со своим
замыканием $\overline{U}$ в $A$, а точная нижняя грань берется по
всем таким $\sigma$, см. предложение 5 из \cite{Kru}.

\medskip

Известно, что при $1\leq p<2$ \begin{equation}\label{maz} {\rm
cap}_{p}\, \mathcal{E}\geqslant 2\pi^{\frac{p}{2}}
\left(\frac{2-p}{p-1}\right)^{p-1}\left[m(C)\right]^{\frac{2-p}{2}}\end{equation}
см., напр., п. 1.4. в \cite{MAZ}.

При $1<p\leq 2$ имеет место оценка
 \begin{equation}\label{krd}
\left({\rm cap}_{p}\,\,
\mathcal{E}\right)\,\ge\,\gamma\,\frac{d(C)^{p}}{m(A)^{p-1}}\,\,,
\end{equation}
где $d(C)$ - диаметр компакта $C$, $m(A)$ - мера Лебега множества
$A$, $\gamma$ - положительная константа, зависящая только от $p\,,$ см. предложение  в \cite {Kru}.

\section{О нижних Q-гомеоморфизмах относительно $p$-модуля }

Следующее понятие мотивировано кольцевым определением
квазиконформности по Герингу, см., например, \cite{Ge$_1$}. Для
заданных областей $D$ и $D'$ в $\Bbb C$, $z_0\in D$, и
измеримой функции $Q:D\to(0,\infty)$, говорят, что гомеоморфизм
$f:D\to D'$ является {\bf нижним $Q$-гомеомор\-физмом относительно $p$-модуля в точке}
$z_0$, если
\begin{equation}\label{nizh}M_p(f\Sigma(z_0,r_1,r_2))\ \geqslant\
\inf\limits_{\varrho\in{\rm ext_p\,adm}\,\Sigma(z_0,r_1,r_2)}\
\int\limits_{R(z_0,r_1,r_2)}\frac{\varrho^p(z)}{Q(z)}\,dm(z)\end{equation} для
каждого кольца $$R(z_0,r_1,r_2)=\{z\in\Bbb  C:r_1<|z-z_0|<r_2\},\
0<r_1<r_2<d_0,$$ где $d_0=dist(z_0, \partial D)$ и $\Sigma(z_0,r_1,r_2)$ обозначает семейство,  всех  окружностей
$C(z_0,r)=\{z\in\mathbb{C}:|\,z-z_0|=r\},\ r\in(r_1,r_2)$.

Прежде чем доказывать основную лемму о нижних $Q-$го\-мео\-морфиз\-мах относительно $p-$мо\-ду\-ля, приведем
вспомогательную лемму из работы \cite{MRSY}.

\medskip

\begin{lemma}\label{thKPR3.1} {\it Пусть $(X,\mu)$ --- измеримое пространство с
конечной мерой $\mu$, $q\in(1,\infty)$, и пусть
$\varphi:X\to(0,\infty)$
--- измеримая функция. Положим
\begin{equation}\label{I}I(\varphi,q)=\inf\limits_{\alpha}
\int\limits_{X}\varphi\,\alpha^q\,d\mu\,,\end{equation} где инфимум берется по всем измеримым функциям
$\alpha:X\to[0,\infty]$ таким, что
\begin{equation}\label{Sal_eq2.1.8}\int\limits_{X}\alpha\,d\mu=1\,.\end{equation}
Тогда
\begin{equation}\label{Sal_eq2.1.9}I(\varphi,q)=\left[\int\limits_{X}\varphi^{-\lambda}\,d\mu\right]
^{-\frac{1}{\lambda}}\,,\end{equation} где
\begin{equation}\label{Sal_eq2.1.10}\lambda=\frac{q'}{q}\,,\qquad\frac{1}{q}+\frac{1}{q'}=1\,,\end{equation}
т.е. $\lambda=1/(q-1)\in(0,\infty)$. Кроме того, инфимум в (\ref{I}) достигается только для функции
\begin{equation}\label{Sal_eq2.1.11}\alpha_0=\gamma\cdot\varphi^{-\lambda}\,,\end{equation} где}
\begin{equation}\label{Sal_eq2.1.12}\gamma=\left(\int\limits_{X}\varphi^{-\lambda}\,d\mu\right)^{-1}\,.\end{equation}
\end{lemma}
\medskip

Ниже приведен  критерий того, что гомеоморфизм в ${\Bbb C}$ является
нижним $Q$-гомеоморфизмом относительно p-модуля.

\begin{lemma}\label{sal1}  {\it Пусть $D$ -- область в $\mathbb{C}$,  $z_0\in D$, и пусть
$Q:D\to(0,\infty)$ -- измеримая функция. Гомеоморфизм $f:D\to\mathbb{C}$ является нижним
$Q-$го\-мео\-морфизмом в точке $z_0$ относительно $p$-мо\-ду\-ля при $p>1$ тогда и только тогда, когда
\begin{equation}\label{Sal_eq2.1.13}M_p(f\Sigma(z_0,r_1,r_2)\geq\int\limits_{r_1}^{r_2}
\frac{dr}{\left(\int\limits_{C(z_0,r)}Q^{\frac{1}{p-1}}(z)\, |dz|\right)^{p-1}},\end{equation}
для всех $0<r_1<r_2<d_0,$ где $d_0=dist(z_0, \partial D)$ и $\Sigma(z_0,r_1,r_2)$ -- семейство всех окружностей
$C(z_0,r)=\{z\in\mathbb{C}:|\,z-z_0|=r\}$, $r\in(r_1,r_2)$. Инфимум в (\ref{nizh}) достигается только для функции}

\begin{equation}\label{Sal_eq2.1.16}\varrho_0(z)=\frac{Q(z)}{\left(\int\limits_{C(z_0,|z-z_0|)}Q^{\frac{1}{p-1}}(z)\, |dz|\right)^{p-1}}\,.\end{equation}

\end{lemma}

\medskip

{\bf Доказательство.}  Для любой  функции $\rho(z)\in  ext_{p}adm\, \Sigma(z_0,r_1,r_2)$
следующая величина
$$A_{\varrho}(r):=\int\limits_{C(z_0,r)}\varrho(z)\,d{\cal A}\neq0\qquad \textrm{п.в.}$$
и является измеримой по параметру $r$, например, по теореме Фубини.
 Таким образом, мы можем требовать равенство
$A_{\varrho}(r)\equiv1$ п.в.
 вместо условия допустимости  (\ref{eq1.2KR}), и

$$\inf\limits_{\varrho\in ext_p\,adm\Sigma(z_0,r_1,r_2)}\ \int\limits_{
R(z_0,r_1,r_2)}\frac{\varrho^p(z)}{Q(z)}\,dm(z)=
\int\limits_{r_1}^{r_2}\left(\inf\limits_{\alpha\in
I(r)}\int\limits_{C(z_0,r)}\frac{\alpha^p(z)}{Q(z)}\,|dz|\right)dr\,,$$ где $I(r)$ --
множество всех измеримых функций $\alpha$ на окружности
$C(z_0,r)$ таких, что
$$\int\limits_{C(z_0,r)}\alpha(z)\,|dz|=1\,.$$

Итак, Лемма  \ref{sal1} следует из Леммы \ref{thKPR3.1} при $X=C(z_0,r)$, $\mu$ -- $1-$мерная длина  на $C(z_0,r)$, $\varphi=\frac{1}{Q}|_{C(z_0,r)}$. Теорема доказана.

\medskip

Таким образом, неравенство (\ref{Sal_eq2.1.13}) является точным для нижних $Q-$го\-мео\-мор\-физмов относительно $p$ -- модуля.

Неравенство (\ref{Sal_eq2.1.13}) можно переписать в несколько ином виде, который иногда будет более удобен для
дальнейшего исследования.

\bigskip

{\bf Теорема.} {\it Пусть   $Q:D\to(0,\infty)$ --- измеримая
функция и $f:D\to D'$  --- нижний
$Q-$ гомеоморфизм в точке $z_0\in D $ относительно $p$-модуля при $p>1$, тогда для любых $0<r_1<r_2< d_0= dist\,(z_0,\partial D)$

$$M_{\frac{p}{p-1}}\left(f\left(\Delta(C_{1}, C_{2}, D)\right)\right)\leq
\left(\int\limits_{r_1}^{r_2}
\frac{dr}{\|Q\|_{\frac{1}{p-1}}(r)}\right)^{-\frac{1}{p-1}}$$
где   $\|Q\|_{\frac{1}{p-1}}(r)=\left(\int\limits_{C(z_0,r)}
Q^{\frac{1}{p-1}}(z)\,|dz|\right)^{p-1}$ }

\begin{proof} Действительно, пусть
$0<r_1<r_2<d(z_0, \partial D)$ и $C_i=C(z_0, r_i),$ $i=1,2.$
Согласно неравенствам Хессе и Цимера, см., напр., \cite{Hes} и
\cite{Zi}, см. также приложения A3 и A6 в \cite{MRSY},
\begin{equation}\label{eqOS6.1aa} M_{\frac{p}{p-1}}\left(f\left(\Delta(C_{1}, C_{2}, D)\right)\right)\ \le
\frac{1}{M_p^{\frac{1}{p-1}}(f\left(\Sigma(z_0,r_1,r_2)\right))}\,,\end{equation}
поскольку
$f\left(\Sigma(z_0,r_1,r_2)\right)\subset\Sigma\left(f(C_1),f(C_2),f(D)\right),$
где $\Sigma(z_0,r_1,r_2)$ обозначает совокупность всех окружностей  с центром в точке
$z_0,$ расположенных между окружностями  $C_1$ и $C_2,$ а
$\Sigma\left(f(C_1),f(C_2),f(D)\right)$ состоит из всех кривых  в $f(D),$ отделяющих $f(C_1)$ и $f(C_2).$ Из
соотношения (\ref{eqOS6.1aa}) по предложению \ref{prOS2.2} получаем,
что
\begin{equation}\label{eqOS6.1b}
M_{\frac{p}{p-1}}\left(f\left(\Delta(C_{1}, C_{2},
D)\right)\right)\ \leq
\left(\int\limits_{r_1}^{r_2}
\frac{dr}{\left(\int\limits_{C(z_0,r)}
Q^{\frac{1}{p-1}}(z)\,|dz|\right)^{p-1}}\right)^{-\frac{1}{p-1}}\,.\end{equation}

\end{proof}

{\bf Лемма.} {\it Пусть $Q:D\rightarrow (0, \infty)$ -- измеримая функция, $Q \in L^{\frac{1}{p-1}}_{loc}(D)$ и   $f:D\to D'$  --- нижний
$Q-$ гомеоморфизм в точке $z_0\in D $ относительно $p$-модуля при $p>1$. Полагаем
$$\eta_0(t)=1/I\cdot\|Q\|_{\frac{1}{p-1}}(z_0,t),$$ где $\|Q\|_{n-1}(z_0,r)$, $r \in (r_1,r_2)$ и $I=I(z_0,r_1,r_2)$
определены в (8) и (10), соответственно. Тогда
$$ I^{-\frac{1}{p-1}}=\int\limits_{R(z_0,r_1,r_2)}
Q^{\frac{1}{p-1}}(z)\cdot\eta_0^\frac{p}{p-1}\left(|z-z_0|\right)dm(z)\leq$$
\begin{equation}\label{2.2}
\leq\int\limits_{R(z_0,r_1,r_2)}
Q^{\frac{1}{p-1}}(z)\cdot\eta^{\frac{p}{p-1}}\left(|z-z_0|\right)dm(z)
\end{equation} для любой измеримой функции
$\eta:(r_1,r_2)\to [0,\infty]$, такой, что
\begin{equation}\label{2.3}
 \int\limits_{r_1}^{r_2}\eta(r)dr=1.  \end{equation} }

\begin{proof}
Если $I=\infty$, то левая часть соотношения (11) 
равна нулю
 и неравенство в этом случае очевидно. Если $I=0$, то
 $\|Q\|_{\frac{1}{p-1}}(z_0,r)=\infty$ для п.в. $r\in (\varepsilon,\varepsilon_0)$ и обе части
неравенства (11) 
равны бесконечности по теореме Фубини и
замечанию 1. Пусть теперь $0<I<\infty$. Тогда
$\|Q\|_{\frac{1}{p-1}}(z_0,r)\neq0$ и $\eta_0(r)\neq\infty$ п.в. в
$(\varepsilon,\varepsilon_0)$. Полагая

 $$\alpha(r)=\eta(r)\cdot\|Q\|_{\frac{1}{p-1}}(z_0,r)$$ и $$
\omega(r)=[\|Q\|_{\frac{1}{p-1}}(z_0,r)]^{-1},
$$ по стандартным соглашениям будем иметь, что
$\eta(r)=\alpha(r)\omega(r)$ п.в. в  $(\varepsilon,\varepsilon_0)$ и что
$$C:=\int\limits_{A}
Q^{\frac{1}{p-1}}(x)\cdot\eta^{\frac{p}{p-1}}\left(|z-z_0|\right)\
dm(z)=\int\limits_{\varepsilon}^{\varepsilon_0}\alpha^{\frac{p}{p-1}}(r)\omega(r)\ dr.$$

Применяя неравенство Иенсена с весом, см. теорему 2.6.2 в
[14], 
к выпуклой функции $\varphi(t)=t^{\frac{p}{p-1}}$, заданной в
интервале $\Omega=(\varepsilon,\varepsilon_0)$, с вероятностной мерой
$$\nu(E)=\frac{1}{I}\int\limits_{E}\omega(r)\ dr,$$
получаем что $$\left(\dashint \alpha^{\frac{p}{p-1}}(r)\omega(r)dr\right)^{\frac{p-1}{p}}\geq \dashint
\alpha(r)\omega(r)\ dr =\frac{1}{I},$$ где мы также использовали
тот факт, что $\eta(r)=\alpha(r)\omega(r)$ удовлетворяет
соотношению (\ref{2.3}). Таким образом, $$C\geq
\frac{1}{I^{n-1}},$$ что и доказывает (12). 
\end{proof}

{\bf Теорема.} {\it Пусть   $Q:D\to(0,\infty)$ --- измеримая
функция и $f:D\to D'$  --- нижний
$Q-$ гомеоморфизм в точке $z_0\in D $ относительно $p$-модуля при $p>1$, тогда для любых $0<r_1<r_2< d_0= dist\,(z_0,\partial D)$

\begin{equation} M_{\frac{p}{p-1}}\left(f\left(\Delta(C_{1}, C_{2}, D)\right)\right)
\leq\int\limits_{R(z_0,r_1,r_2)}
Q^{\frac{1}{p-1}}(z)\cdot\eta^{\frac{p}{p-1}}\left(|z-z_0|\right)dm(z)
\end{equation} для любой измеримой функции
$\eta:(r_1,r_2)\to [0,\infty]$, такой, что
\begin{equation}\label{2.3}
 \int\limits_{r_1}^{r_2}\eta(r)dr=1.  \end{equation} }

\section{ Конечная липшицевость  нижних $Q$-гомеоморфизмов
относительно $p$-модуля.}

В дальнейшем рассматриваются открытые множества $\Omega$ в
$\mathbb{C}$ и непрерывные отображения
$f:\Omega\to\mathbb{C}$. Для $f:\Omega\to\mathbb{C}$ и
$x\in\Omega\subseteq\mathbb{C}$, положим
\begin{equation}\label{eq3.1.3}L(z,f)=\limsup_{\zeta\to
z}\frac{|f(\zeta)-f(z)|}{|\zeta-z|}\,.\end{equation}
 Будем говорить, что отображение
$f:\Omega\to\mathbb{C}$ является {\it конечно липшицевым}, если
\begin{equation}\label{eq3.1.5}L(z,f)<\infty\end{equation} для всех $z\in\Omega$,
Очевидно, что каждое липшицево отображение является конечно
липшицевым.

Ниже приведена теорема   о  достаточном условии локальной
липшицевости в точке для нижних  $Q$-гомеоморфизмов  относительно
$p$-модуля при $p>2$.

{\bf Лемма  2.}  {\it Пусть  $D$ и $D'$ -- области в $\mathbb{C}$,
 $Q:D\,\rightarrow\,[0,\,\infty]\,$ -- локально
интегрируемая  функция и   $f:D\to D'$ -- нижний $Q$-гомеоморфизм
относительно $p$-модуля в точке $z_0\in D$ с условием

$$ Q_0=\limsup\limits_{\varepsilon\to 0}\left(
\dashint_{B(z_0,\varepsilon) } Q^{\frac{1}{p-1}}(z)\, dm(z)\right)^{p-1}<\infty.
$$

Тогда при    $p>2$ имеем  $$L(z_0,f)=\limsup\limits_{z\to
z_0}\frac{|f(z)-f(z_0)|}{|z-z_0|}\leq \lambda_{p}\,
Q^{\frac{1}{p-2}}_{0},$$ где  $\lambda_{p}$ - положительная
постоянная, зависящая только  $p$. }

{\it Доказательство.} Рассмотрим сферическое кольцо
$R=R(z_0,\varepsilon_1, \varepsilon_2)$ с $0<\varepsilon_1<\varepsilon_2$ такое,
что $R(z_0,\varepsilon_1, \varepsilon_2)\subset D$. Тогда
$\left(fB\left(z_0,\varepsilon_2\right),\overline{fB\left(z_0,\varepsilon_1\right)}\right)$
-- кольцевой конденсатор в   $D'$ и, согласно (\ref{EMC}), имеем
равенство
$$ {\rm cap}_{\frac{p}{p-1}}\ (fB(z_0,\varepsilon_2),\overline{fB(z_0
,\varepsilon_1)})=M_{\frac{p}{p-1}}(\triangle(\partial
fB(z_0,\varepsilon_2),\partial f B(z_0,\varepsilon_1);fR)$$ а ввиду
гомеоморфности  $f,$ равенство
$$\triangle\left(\partial
fB\left(z_0,\varepsilon_2\right),\partial
fB\left(z_0,\varepsilon_1\right);fR\right)=f\left(\triangle\left(\partial
B(z_0,\varepsilon_2) ,\partial
B(z_0,\varepsilon_1);R\right)\right).$$

Рассмотрим функцию
$$ \eta(t)\,=\,\left
\{\begin{array}{rr} \frac{1}{\varepsilon_2-\varepsilon_1}, &  \ t\in (\varepsilon_1,\varepsilon_2) \\
0, & \ t\in \Bbb{R}\setminus (\varepsilon_1,\varepsilon_2).
\end{array}\right.
$$
В  силу определения  кольцевого  $Q$-гомеоморфизма
  относительно
$p$-модуля, замечаем, что

\begin{equation}\label{eq100} {\rm cap}_{\frac{p}{p-1}}\
(fB(z_0,\varepsilon_2),\overline{fB(z_0,\varepsilon_1)}) \leq (\varepsilon_2-\varepsilon_1)^{-\frac{p}{p-1}}\int\limits_{R(z_0,\varepsilon_1,
\varepsilon_2)} Q^{\frac{1}{p-1}}(z)\ dm(z)\ .\end{equation} Далее, выбирая
$\varepsilon_1=2\varepsilon$ и $\varepsilon_2=4\varepsilon$, получим

\begin{equation}\label{eq101}{\rm cap_{\frac{p}{p-1}}}\ (fB(z_0,4\varepsilon),f\overline{B(z_0,2\varepsilon)})\le\,
(2\varepsilon)^{-\frac{p}{p-1}}\int\limits_{B(z_0,4\varepsilon)}Q^{\frac{1}{p-1}}(z)\,dm(z)
\end{equation}
С другой стороны,  в силу  неравенства (\ref{maz}) вытекает оценка

\begin{equation}\label{eq102} {\rm cap_{\frac{p}{p-1}}}\ (fB(z_0,4\varepsilon),f\overline{B(z_0,2\varepsilon)})
\ge C_{p}\left[m(fB(z_0,2\varepsilon))\right]^{\frac{p-2}{2(p-1)}}
\end{equation}
где    $C_{p}$  -- положительная константа, зависящая только от  $p.$

Комбинируя   (\ref{eq101}) и (\ref{eq102}), получаем, что
\begin{equation}\label{eq4.2} \frac{m(fB(z_0,2\varepsilon))}{m(B(z_0,2\varepsilon))}\leqslant c_{p}\,\left[
\dashint_{B(z_0,4\varepsilon)}\, Q^{\frac{1}{p-1}}(z)\, dm(z)
\right]^{\frac{2(p-1)}{p-2}}\,,\end{equation} где $c_{p}$ -
положительная постоянная зависящая только от  $p$.

Далее, выбирая в (\ref{eq100}) $\varepsilon_1=\varepsilon$ и
$\varepsilon_2=2\varepsilon$, получим
\begin{equation}\label{eq91}{\rm cap_{{\frac{p}{p-1}}}}\ (fB(z_0,2\varepsilon),f\overline{B(z_0,\varepsilon)})\le\,
\varepsilon^{-{\frac{p}{p-1}}}\int\limits_{B(z_0,2\varepsilon)}Q^{{\frac{1}{p-1}}}(z)\,dm(z)\,.
\end{equation}
С другой стороны, в силу неравенства (\ref{krd}), получаем
\begin{equation}\label{eq10*!} {\rm cap_{\frac{p}{p-1}}}\ (fB(z_0,2\varepsilon),f\overline{B(z_0,\varepsilon)})
\ge\widetilde{C}_{p}\frac{d^{\frac{p}{p-1}}(fB(z_0,\varepsilon))}{m^{{\frac{1}{p-1}}}(fB(z_0,2\varepsilon))}
\end{equation}
где    $\widetilde{C}_{p}$  -- положительная константа, зависящая только  $p.$

Комбинируя   (\ref{eq91}) и (\ref{eq10*!}), получаем, что
$$\frac{d(fB(z_0,\varepsilon))}{\varepsilon}\le\gamma_{p}
\left(\frac{m(fB(z_0,2\varepsilon))}{m(B(z_0,2\varepsilon))}
\right)^{\frac{1}{p(p-1)}}\left( \dashint_{B(z_0,2\varepsilon)}
Q^{\frac{1}{p-1}}(z)\, dm(z)\right)^{\frac{p-1}{p}}\,,$$ где $\gamma_{p}$--
положительная константа, зависящая только от $p.$

Эта оценка вместе с  (\ref{eq4.2}) дает неравенство
$$\frac{d(fB(z_0,\varepsilon))}{\varepsilon}\leq\lambda_{p}\left(\dashint_{B(z_0,4\varepsilon)} Q^{\frac{1}{p-1}}(z)\,
dm(z)\right)^{\frac{2(p-1)}{p(p-2)}}\left[
\dashint_{B(z_0,2\varepsilon)}\, Q^{\frac{1}{p-1}}(z)\, dm(z)
\right]^{\frac{p-1}{p}}\,.$$ Переходя к верхнему пределу при
$\varepsilon\to 0$  немедленно вытекает заключение леммы
$$L(z_0,f)=\limsup\limits_{z\to
z_0}\frac{|f(z)-f(z_0)|}{|z-z_0|}\leq\limsup\limits_{\varepsilon\to
0}\frac{d(fB(z_0,\varepsilon))}{\varepsilon}\leq \lambda_{p}\,
Q^{\frac{1}{p-2}}_{0} ,
$$
где  $\lambda_{p}$ - положительная  постоянная, зависящая только
от  $p$.

\medskip

{\bf Теорема  3.}  {\it Пусть  $D$ и $D'$ -- области в $\mathbb{C}$,
 $Q:D\,\rightarrow\,[0,\,\infty]\,$ -- локально
интегрируемая  функция и   $f:D\to D'$ -- нижний $Q$-гомеоморфизм
относительно $p$-модуля $D$ с условием
$$
\limsup\limits_{\varepsilon\to 0}\left(
\dashint_{B(z_0,\varepsilon) } Q^{\frac{1}{p-1}}(z)\, dm(z)\right)^{p-1}<\infty \ \ \ \    \forall x_0\in D.
$$
Тогда при    $p>2$
гомеоморфизм  $f$ является конечно липшицевым. }

\medskip

{\bf Замечание.} В соответствии с леммой 10.6 в \cite{MRSY} конечно
липшицевые отображения обладают $N$-свойством относительно
хаусдорфовых мер и, таким образом, являются абсолютно непрерывными
на кривых.

\section{ Искажение площади круга.}
\setcounter{equation}{0}

В этом разделе получена  оценка площади образа круга при нижних
$Q$-гомеоморфизмах относительно $p$-модуля. Впервые оценка площади
образа круга при квазиконформных отоб\-ра\-же\-ниях встречается в
монографии  Лаврентьева М.А., см. \cite{La}.

\medskip

{\bf Теорема 3.1} {\it Пусть  $f$ -- нижний
$Q$-гомеоморфизм $\mathbb{B}$ в $\mathbb{B}$ относительно
$p$-модуля. Тогда при $p>2$ имеет место оценка
\begin{equation}\label{eqks*}m(fB_r)\leqslant\ \pi\cdot \left(1 +(2\pi)^{p-1}(p-2)\,
\int\limits_{r}^{1}
\frac{dt}{\|Q\|_{\frac{1}{p-1}}(t)} \right)^
{\frac{2}{2-p}}
,\end{equation} а при $p=2$
\begin{equation}\label{eqks**}m(fB_r)\leqslant \pi
\exp\left\{-4 \pi\int\limits_{r}^{1}\frac{dt}{\|Q\|_{1}(t)}\right\}\,.\end{equation}
}

\medskip

\begin{proof} Рассмотрим сферическое кольцо
$R_{t}=\{x\in\mathbb{B}^n:t<|x|<t+\triangle t\}$. Пусть
$(A_{t+\triangle t},C_t)$ -- конденсатор, где $C_t=\overline{B_t},
A_{t+\triangle t}=B_{t+\triangle t}$. Тогда $(fA_{t+\triangle
t},fC_t)$ -- кольцевой конденсатор в $\mathbb{C}$ и согласно
(\ref{eqks2.7}) имеем \begin{equation}\label{eqks1.8}{\rm
cap}_{\frac{p}{p-1}}(fA_{t+\triangle t},fC_t)=M_{\frac{p}{p-1}}(\triangle(\partial
fA_{t+\triangle t},\partial fC_t;fR_{t})).\end{equation} В силу
неравенства (\ref{eqks2.8}) получим
\begin{equation}\label{eqks1.9}{\rm cap}_{\frac{p}{p-1}}\left(fA_{t+\triangle
t}, fC_t\right)\geqslant\frac{\left(\inf
m_{1}\,\sigma\right)^{\frac{p}{p-1}}}{m\left(fA_{t+\Delta t}\setminus
fC_t\right)^{\frac{1}{p-1}}},\end{equation} где $ m_{1}\,\sigma$ --
$1$-мерная мера Лебега $C^{\infty}$-многообразия $\sigma$,
являющегося границей $\sigma=\partial U$ ограниченного открытого
множества $U$, содержащего $fC_t$ и содержащегося вместе со своим
замыканием $\overline{U}$ в $fA_{t+\triangle t}$, а точная нижняя
грань берется по всем таким $\sigma$.

С другой стороны, в силу леммы 2.1, имеем \begin{equation}
\label{eqks1.10}{\rm cap}_{\frac{p}{p-1}}\left(fA_{t+\triangle t},
fC_t\right)\leqslant \left(\int\limits_{t}^{t+\triangle t}
\frac{dr}{\|Q\|_{\frac{1}{p-1}}(r)}\right)^{-\frac{1}{p-1}}\,.
\end{equation}
Комбинируя  неравенства (\ref{eqks1.9}) и (\ref{eqks1.10}), получим
$$\frac{\left(\inf
m_{1}\,\sigma\right)^{\frac{p}{p-1}}}{m\left(fA_{t+\Delta t}\setminus
fC_t\right)^{\frac{1}{p-1}}}\leqslant \left(\int\limits_{t}^{t+\triangle t}
\frac{dr}{\|Q\|_{\frac{1}{p-1}}(r)}\right)^{-\frac{1}{p-1}}\,.$$ Далее, воспользовавшись
изопериметрическим неравенством
$$\inf m_{1}\, \sigma\geqslant 2\cdot\pi^{\frac{1}{2}}
\left(m(fC_t)\right)^{\frac{1}{2}},$$ получим

\begin{equation}\label{eqks1.11}2\cdot\pi^{\frac{1}{2}}
\left(m(fC_t)\right)^{\frac{1}{2}}\leq
\left(\frac{m\left(fA_{t+\Delta t}\setminus
fC_t\right)}{\int\limits_{t}^{t+\triangle t}
\frac{dr}{\|Q\|_{\frac{1}{p-1}}(r)}}\right)^{\frac{1}{p}}.\end{equation}

Полагая $\Phi(t):=m\left(fB_t\right)$,  из соотношения
(\ref{eqks1.11}) имеем, что

\begin{equation}\label{eqks1.12}2\cdot\pi^{\frac{1}{2}}\Phi^{\frac{1}{2}}(t)\leq
 \left(\frac{\frac{\Phi(t+\Delta
t)-\Phi(t)}{\Delta t}}{\frac{1}{\Delta t} \int\limits_{t}^{t+\triangle t}
\frac{dr}{\|Q\|_{\frac{1}{p-1}}(r)}
}\right)^{\frac{1}{p}} \,.\end{equation}

Заметим, что в   силу теоремы 2.3 и гомеоморфности отображения $f$,
$$\|Q\|^{-1}_{\frac{1}{p-1}}(r)\in L^{1}_{loc}
(0,1)$$

Устремляя в неравенстве (\ref{eqks1.12}) $\Delta t$ к нулю, и
учитывая монотонное возрастание функции $\Phi(t)$ по $t\in(0,1)$ и
равенство $\omega_{n-1}=n\Omega_n$, для п.в. $t$ имеем существование
производной $\Phi'(t)$ и

\begin{equation}\label{eqks1.14} 2^p \pi^{\frac{p}{2}} \|Q\|^{-1}_{\frac{1}{p-1}}(t)\leqslant \frac{\Phi'(t)}{\Phi^{\frac{p}{2}}(t)}
.\end{equation}

Рассмотрим неравенство (\ref{eqks1.14}) при $p>2$. Интегрируя обе
части неравенства по $t\in [r,1]$ и учитывая, что
$$\int\limits_{r}^{1}\frac{\Phi'(t)}{\Phi^{\frac{p}{2}}(t)}\,dt\leqslant
\frac{2}{2-p}\left(\Phi^{\frac{2-p}{2}}(1)-\Phi^{\frac{2-p}{2}}(r)\right),$$
см., напр., теорему  IV. 7.4 в \cite{Sa}, получим
\begin{equation}\label{eqks1.15}2^{p} \pi^{\frac{p}{2}}\int\limits_{r}^{1} \frac{dt}{\|Q\|_{\frac{1}{p-1}}(t) } \leqslant \frac{2}{2-p}\left(\Phi^{\frac{2-p}{2}}(1)-\Phi^{\frac{2-p}{2}}(r)\right)
.\end{equation}
Из неравенства (\ref{eqks1.15}) получаем, что
$$\Phi(r)\leqslant\left(\Phi^{\frac{2-p}{2}}(1)-2^{p-1}(2-p)\pi^{\frac{p}{2}}\,
\int\limits_{r}^{1}
\frac{dt}{\|Q\|_{\frac{1}{p-1}}(t)} \right)^
{\frac{2}{2-p}}.$$ Наконец,  учитывая, что
$m(f\mathbb{B})\leq \pi$, приходим к (\ref{eqks*}).

Осталось рассмотреть случай $p=2$. В этом случае неравенство
(\ref{eqks1.14}) примет вид:

\begin{equation}\label{eqks1.14uii} \frac{4 \pi}{\|Q\|_{1}(t)}\leqslant \frac{\Phi'(t)}{\Phi(t)}
.\end{equation}

Интегрируя обе части неравенства (\ref{eqks1.14uii}) по $t\in[r,1]$,
учитывая, что
$$\int\limits_{r}^{1}\frac{\Phi'(t)}{\Phi(t)}dt\leqslant\ln\frac{\Phi(1)}{\Phi(r)},$$
см., напр., теорему  IV. 7.4 в \cite{Sa}, получим
$$4 \pi\int\limits_{r}^{1}\frac{dt}{\|Q\|_{1}(t)}\leqslant\ln\frac{\Phi(1)}{\Phi(r)}.$$ И, следовательно, имеем
$$\exp\left\{ 4 \pi\int\limits_{r}^{1}\frac{dt}{\|Q\|_{1}(t)}\right\}\leqslant
\frac{\Phi(1)}{\Phi(r)},$$ а потому
$$\Phi(r)\leqslant\Phi(1)\cdot
\exp\left\{-4\pi\int\limits_{r}^{1}\frac{dt}{\|Q\|_{1}(t)}\right\}
\,,$$ что
и приводит нас к неравенству (\ref{eqks**}) поскульку $\Phi(1)\leq
\pi$.

 \end{proof}

\bigskip

\medskip
\section{ Поведение в точке.}
\setcounter{equation}{0}

Теорема, приведенная в предыдущей секции, позволяет  нам также
описать асимптотическое поведение кольцевых $Q$-гомеоморфизмов
относительно $p$-модуля в нуле.

\medskip

{\bf Предложение 4.1.} {\it Пусть $f$ -- гомеоморфизм
 $\mathbb{C}$ в $\mathbb{C}$, $f(0)=0$. Если
\begin{equation}\label{eqks1.18}m(fB_r)\leqslant\pi\,\mathcal{R}^{2}(r),\end{equation} то
\begin{equation}\label{eqks1.19}\liminf\limits_{z\to 0}\frac{|f(z)|}{\mathcal{R}(|z|)}\leqslant1.\end{equation}}

\begin{proof} Положим $l_f(r)=\min\limits_{|z|=r}|f(z)|$.
Тогда, учитывая, что $f(0)=0$, получаем
$\pi\,l_{f}^{2}(r)\leqslant m(fB_r)$ и, следовательно,
\begin{equation}\label{eqks1.20}l_{f}(r)\leqslant\left(\frac{m(fB_r)}{\pi}\right)^{\frac{1}{2}}.\end{equation}
Таким образом, учитывая  неравенство  (\ref{eqks1.18}), имеем
$$\liminf\limits_{z\to 0}\frac{|f(z)|}{\mathcal{R}(|z|)}=\liminf\limits_{r\to0}\frac{l_{f}(r)}{\mathcal{R}(r)}
\leqslant\liminf\limits_{r\to0}\left(\frac{m(fB_r)}{\pi}\right)^{\frac{1}{2}}\cdot\frac{1}{\mathcal{R}(r)}\leqslant1.$$
Предложение доказано. \end{proof}

\medskip

Комбинируя теорему  3.1 и предложение 4.1 с функцией
$$\mathcal{R}(r)=\left(1 +(2\pi)^{p-1}(p-2)\,
\int\limits_{r}^{1}
\frac{dt}{\|Q\|_{\frac{1}{p-1}}(t)} \right)^
{-\frac{1}{p-2}}  $$
при $p>2$  и
$$\mathcal{R}(r)=\exp\left\{-2 \pi\int\limits_{r}^{1}\frac{dt}{\|Q\|_{1}(t)}\right\}$$
при $p=2$, получаем следующий результат.

\medskip

{\bf Теорема 4.2.} Пусть $f$ -- кольцевой $Q$-гомеоморфизм
  относительно
$p$-модуля $\mathbb{B}$ в $\mathbb{B}$,  и
$f(0)=0$. Тогда при $p>2$ имеет место оценка
\begin{equation}\label{eqks1.21}\liminf\limits_{z\to 0}\,|f(z)|\cdot
\left(1 +(2\pi)^{p-1}(p-2)\,
\int\limits_{|z|}^{1}
\frac{dt}{\|Q\|_{\frac{1}{p-1}}(t)} \right)^
{\frac{1}{p-2}}\leqslant1,\end{equation}
а при $p=2$
\begin{equation}\label{eqks1.22}\liminf\limits_{z\to 0}\,|f(z)|\cdot
\exp\left\{2 \pi\int\limits_{|z|}^{1}\frac{dt}{\|Q\|_{1}(t)}\right\}\leqslant1
\,.\end{equation}

\medskip

{\bf Следствие 4.3.} Пусть $f$ -- нижний  $Q$-гомеоморфизм
  относительно
$p$-модуля $\mathbb{B}$ в $\mathbb{B}$  и
$f(0)=0$. Тогда при $p>2$ имеет место оценка
\begin{equation}\label{eqks1.21}\liminf\limits_{z\to 0}\,|f(z)|\cdot
\left(\int\limits_{|z|}^{1}
\frac{dt}{\|Q\|_{\frac{1}{p-1}}(t)} \right)^
{\frac{1}{p-2}}\leqslant(2\pi)^{1-p}(p-2)^{\frac{1}{2-p}}.\end{equation}

\section{ О взаимосвязи с классами Соболева }

\begin{theorem}\label{thKPR3.1} {\it Пусть $D$ и $D'$ -- области в $\mathbb{C}$ и $f:D\to D' $ -- гомеоморфизм класса Соболева $W^{1,1}_{\rm loc}$ с конечным искажением. Тогда $f$ является
нижним $Q$-гомеоморфизмом относительно $p$-модуля  в каждой точке $z_0\in D$ с
$Q(z)=K_{p}(f,z)$ .} \end{theorem}

\begin{proof} Пусть $\mathfrak{B}$ -- множество (борелевское!) всех
точек $z$ из $D$, в которых $f$ имеет полный дифференциал с $J_{f}(z)\neq0$. Известно, что $\mathfrak{B}$ можно представить
в виде объединения счетного набора борелевских множеств $\mathfrak{B}_l$, $l=1,2,\ldots$, таких, что $f_l=f|_{\mathfrak{B}_l}$ являются
билипшицевыми гомеоморфизмами, см., например, лемму 3.2.2 в \cite{Fe_1}. Без потери общности можно считать, что
$\mathfrak{B}_l$ попарно не пересекаются. Пусть $\mathfrak{B}_{*}$ -- множество всех точек $z\in D$, где $f$ имеет полный дифференциал
с $f_{z}=0=f_{\bar z}$.

Заметим, что по известной теореме Геринга--Лехто--Меньшова множество $\mathfrak{B}_0=D\setminus(\mathfrak{B}\cup \mathfrak{B}_{*})$ имеет нулевую
меру Лебега в ${\Bbb C}$, см. \cite{GL} и \cite{Me}. Следовательно, по теореме 2.11 в \cite{KR$_2$}, см. также
лемму 9.1 в \cite{MRSY}, $l(\gamma\cap \mathfrak{B}_0)=0$ для п.в. штриховых линий $\gamma$ в $D$. Покажем также, что
$l(f(\gamma)\cap f(\mathfrak{B}_0))=0$ для п.в. окружностей $\gamma$ с центром в точке $z_0$.

Последнее вытекает из абсолютной непрерывности $f$ на замкнутых подду\-гах $\gamma\cap D$ для п.в. окружностей
$\gamma$. Действительно, класс $W^{1,1}_{\rm loc}$ является инвариантным относительно локально
квазиизометрических преобразований независимой переменной, см., например, теорему 1.1.7 в \cite{Maz}, и функции
из $W^{1,1}_{\rm loc}$ абсолютно непрерывны на линиях, см., например, теорему 1.1.3 в \cite{Maz}. Применяя, к
примеру, преобразование координат $\log(z-z_0)$, мы приходим к абсолютной непрерывности на п.в. окружностях
$\gamma$ с центром в точке $z_0$.

Таким образом, $l(\gamma_{*}\cap f(\mathfrak{B}_0))=0$, где $\gamma_{*}=f(\gamma)$, для п.в. окружностей $\gamma$ с центром
в точке $z_0$. Пусть теперь $\varrho_*\in{\rm adm}\,f(\Gamma)$, $\varrho_*\equiv0$ вне $f(D)$, где $\Gamma$ --
совокупность всех штриховых линий, образованных пересечениями всех окружностей $\gamma$ с центром в точке $z_0$.
Пусть $\varrho\equiv0$ вне $D$ и
$$\varrho(z)\colon=\varrho_*(f(z))\left(|f_{z}|+|f_{\bar z}|\right)\qquad
\text{для\ п.в.}\ \ z\in D.$$

Рассуждая кусочно на $\mathfrak{B}_l$, мы имеем по теореме 3.2.5 из \cite{Fe_1} (при $m=1$), что
$$\int\limits_{\gamma}\varrho\,ds\ \geqslant\
\int\limits_{\gamma_{*}}\varrho_{*}\,ds_{*}\ \geqslant\ 1\qquad \text{для\ п.в.}\ \ \gamma\in\Gamma,$$ поскольку
$l(f(\gamma)\cap f(\mathfrak{B}_0))=0$, а также $l(f(\gamma)\cap f(\mathfrak{B}_{*}))=0$ для $p$-п.в. $\gamma\in\Gamma$ ввиду абсолютной
непрерывности $f$ на $p$-п.в. $\gamma\in\Gamma$. Следовательно, $\varrho\in{\mathrm{ext_p\,adm}}\,\Gamma$.

С другой стороны, еще раз рассуждая кусочно на $\mathfrak{B}_l$, мы имеем неравенство
$$\int\limits_{D}\frac{\varrho^p(x)}{K_{O,p}(z,f)}\ dm(z)\ \leqslant\
\int\limits_{f(D)}\varrho^p_*(w)\ dm(w),$$ поскольку $\varrho(z)=0$ на $\mathfrak{B}_{*}$. Следовательно, мы получаем, что
$$M_p(f\Gamma)\ \geqslant\ \inf\limits_{\varrho\in
{\mathrm{ext_p\,adm}}\,\Gamma}\int\limits_{D}\frac{\varrho^p(z)}{K_{p}(z,f)}\ dm(z)\,,$$ т.е., $f$ действительно
является нижним $Q-$гомеоморфизмом с $Q(z)=K_{p}(z,f)$  относительно $p$-модуля.  \end{proof}

КОНТАКТНАЯ ИНФОРМАЦИЯ

\medskip
\noindent{{\bf Салимов Руслан Радикович}
\\Институт
прикладной математики и механики НАН Украины \\
83 114 Украина, г. Донецк, ул. Розы Люксембург, д. 74, \\отдел
теории функций, раб. тел. (380) -- 62 -- 311 01 45, \\ 
e-mail: brusin2006@rambler.ru; ruslan623@yandex.ru}

\end{document}